\documentclass{article}
\usepackage{graphicx} 
\usepackage{color}
\usepackage{amsmath}
\usepackage{amsthm}
\usepackage{comment}
\usepackage{pgfgantt}
\usepackage{mathtools}
\usepackage{graphics}
\usepackage{colortab}
\usepackage{color}
\usepackage{csvsimple}
\usepackage[backend=biber,style=numeric,autocite=plain,sorting=nyt,maxbibnames=99]{biblatex}
\newtheorem{theorem}{Theorem}
\newtheorem{lemma}{Lemma}
\newtheorem{definition}{Definition}

\newtheorem{remark}{Remark}

\title{Avoiding Deadlocks via Weak Deadlock Sets}
\date{}

\author{
  Gianpaolo Oriolo %
  \thanks{DICII, Universit\`a degli Studi ``Tor Vergata'',
  Roma, Italy \texttt{oriolo@disp.uniroma2.it}}
  \and Anna Russo Russo%
   \thanks{DICII, Universit\`a degli Studi ``Tor Vergata'',
  Roma, Italy \texttt{anna.russorusso23@gmail.com}}
}

\begin{document}

\maketitle
\begin{abstract}
	A deadlock occurs in a network when two or more items prevent each other from moving and are stalled. In a general model, items are stored at vertices and each vertex $v$ has a buffer with $b(v)$ slots. Given a route for each item toward its destination, the Deadlock Safety Problem asks whether the current state is {\em safe}, i.e., it is possible to deliver each item at its destination, or is {\em bound to deadlock}, i.e., any sequence of moves will end up with a set of items stalled. While when $b \geq 2$ the problem is solvable in polynomial time building upon a nice characterization of \textsc{yes/no}-instances, it is \textsc{np}-hard on quite simple graphs as grids when $b=1$ and on trees when $b\leq 3$.

We improve on these results by means of two new tools, weak deadlock sets and wise states. We show that for general networks and $b$ a state that is wise and without weak deadlock sets -- this can be recognized in polynomial time -- is safe: this is indeed a strengthening of the result for $b\geq 2$. We sharpen this result for trees, where we show that a wise state is safe if and only if it has no weak deadlock set. That is interesting in particular in the context of  rail transportation where networks are often single-tracked and deadlock detection and avoidance focuses on local sub-networks, mostly with a tree-like structure.

We pose some research questions for future investigations.
\end{abstract}


\section{Introduction}\label{intro}
A deadlock occurs in a network when two or more items, e.g. trains, prevent each other from moving and are stalled. A similar concept was developed in the context of rail transportation~\cite{dalsasso_2021, dalsasso_2022, li_2014, luteberget_2021, pachl_2011, petersen_1983}, store-and-forward networks~\cite{arbib_1988, blazewicz_1994}, computer systems~\cite{coffman_1971, holt_1972}, automated manufacturing systems~\cite{fanti_1997,viswanadham_1990}.

In one of the first paper dealing with deadlocks, in the area of computer systems, Coffman, Elphick, and  Shoshani~\cite{coffman_1971} distinguish three types of strategies for handling deadlocks: {\em deadlock detection and recovery}, {\em deadlock prevention} and {\em deadlock avoidance}. Algorithms of deadlock detection and recovery let deadlocks occur in a system but take action as soon as that happens. Algorithms for deadlock prevention impose restrictions on the system evolution in order to rule out deadlocks and are usually designed around strict conditions for avoiding deadlock formation. Algorithms for deadlock avoidance also exploit information on the current state in the network in order to prevent deadlock: in particular, while deadlock prevention algorithms may unnecessarily impede harmless state transitions, deadlock avoidance algorithms are maximally permissive and guarantee system progress whenever it is possible. 

In this paper, we focus on a problem, the {\em Deadlock Safety Problem} that was introduced in~\cite{arbib_1988}, and falls in the area of deadlock avoidance algorithms. In this model, items are placed at the vertices of a network $G(V,E)$, each vertex $v$ has a buffer that can accommodate up to $b(v)$ items, and each item has to be delivered at its destination following a route that is given. The Deadlock Safety Problem asks whether the current state is {\em safe}, i.e., it is possible to deliver each item at its destination, or vice versa is {\em bound to deadlock}, i.e., any possible sequence of movements will end up in a state where a set of items are stalled.

The complexity picture for the Deadlock Safety Problem is quite sensitive to the number of slots in each buffer: we call this number the {\em size} of the buffer. The problem is hard on quite simple graphs when this size is small, but easy on {\em any} graph as soon as the size is large enough. 
Namely, it is shown in~\cite{arbib_1988} that the problem is \textsc{np}-hard on trees when $b\leq 3$ and is hard on grids and other simple graphs already when $b=1$. However, things are quite different when buffers are ``large'': in their beautiful paper~\cite{blazewicz_1994}, Blazewicz et al. show that when $b\geq 2$ the problem can be solved in polynomial time. In this case, a state is safe if and only if there is no subset of items that is already stalled; moreover, that can be detected in polynomial time building upon the definition of {\em strong deadlock set} of vertices (formal definitions will be given in the following).

The results in~\cite{blazewicz_1994} characterize the Deadlock Safety Problem when $b\geq 2$. Unfortunately while large buffers are common in contexts such as store-and-forward networks and computer systems, that is often not the case for rail transportation networks. E.g., in 2015 only about 11 percent of the US railroad network was double-tracked~\cite{Hicks_2015}. Moreover, still in the context of railroad transportation, deadlocks are managed at a local scale where networks have a tree-like structure, see e.g.~\cite{dalsasso_2021, dalsasso_2022}.

\smallskip
{\bf Our contribution.} 
In order to go beyond the previous results we need to be able to handle vertices with $b(v) = 1$. Our main insight is that when there are buffers of different sizes, i.e., some with $b(v) \geq 2$ and some with $b = 1$, items should only {\em transit} on the latter and not {\em stay}. This led us to introduce two new tools, weak deadlock sets and wise states. Weak deadlock sets are easy to recognize and generalize strong deadlock sets to include some vertices $v$ with $b(v)=1$ in the set; it is indeed the case that when $b\geq 2$ any weak deadlock set is also strong. A state $\sigma$ is wise if $\sigma(v) = 0$ for each vertex $v$ with $b(v) =1$; hence when $b\geq 2$ any state is wise. We show that, for general networks and $b$, any state that is wise and without weak deadlock sets is safe. Note that, when $b\geq 2$, our results reduce to that in~\cite{blazewicz_1994} and it is indeed the case that our proof builds upon a careful analysis and extensions of the proofs in~\cite{blazewicz_1994}. We also prove that for trees with general $b$ a wise state is bound to deadlock if and only it has already a weak deadlock set. Therefore, on trees if for the current state there are no items on vertices $v$ with $b(v)=1$ and no weak deadlock sets, then it is possible to free the network.

 Finally, we pose a a few research questions, motivated by our results and proofs, that we believe should be addressed in order to further improve the complexity picture for the Deadlock Safety Problem. 

\smallskip
The paper is organized as follows. We close this section with a few definitions and examples. We devote Section~\ref{literature} to the literature and Section~\ref{bovetpaper} to provide more details about the results in~\cite{blazewicz_1994} for the case $b\geq 2$. We devote Section~\ref{wds} to our first new tool, the weak deadlock set, and in particular to show that on trees any state with a weak deadlock set is bound to deadlock. We devote Section~\ref{ws} to our other new tool, the wise state, and show that a wise state without weak deadlock sets is safe. Our last section is devoted to some concluding remarks and open questions.

\subsection{Definitions}\label{sec:def}
In this section we formally define the problem we are interested in: the {\em Deadlock Safety Problem}. While this problem was originally defined in~\cite{arbib_1988}, we prefer to state it from scratch, building upon our own definitions. In the following, for $k\in {\cal Z}_+$, we let $[k]: = \{1, \ldots, k\}$. In this model, items are placed at the vertices of an undirected network $G(V,E)$. Each vertex $v\in V$ has then a buffer with $b(v)$ slots that can accommodate up to $b(v)$ items: we assume that each item fits into a single slot and that vice versa each slot cannot host more than one item. We are then given: a pair $(G, b)$ where $G$ is an undirected network with vertices $V(G)$ and edges $E(G)$ and $b: V(G)\mapsto {\cal Z}_+$ is the size of each buffer at vertex $v$;

In the following, when there is no risk of confusion, we let $V:=V(G)$ and $E:=E(G)$. Note also that we usually refer to {\em vertices} and {\em edges} for undirected networks and to {\em nodes} and {\em arcs} for directed networks.

\smallskip
We are also given initial placements for items at vertices of $V$ through a map $\sigma: V\mapsto {\cal Z}_+$; therefore, for each $v\in V$, $\sigma(v)$ items are placed at the buffer of $v$ (we of course assume that $\sigma(v)\leq b(v)$). We are also given, for each $v\in V$ and each item ``sitting'' at $v$, a {\em destination}, which is a vertex in $V$, and a {\em route}, that is a path of $G$ from $v$ to the destination. The combination of placements and routes is called a {\em state}. More formally, a state is defined as follows: 

\begin{definition}[State]\label{follower}
A state is a pair $(\sigma, {\cal R})$ such that the followings hold:
\begin{itemize}
\item $\sigma: V\mapsto {\cal Z}_+$ is such that $\sigma(u) \leq b(u)$, for each $u\in V$;
\item ${\cal R}:= \bigcup_{u\in V}{\cal R}(u)$, where each ${\cal R}(u)$ is a collection of $\sigma(u)$ simple paths in $G$, each starting at $u$. Note that we allow for multiple copies of a same path in each ${\cal R}(u)$, i.e., each ${\cal R}(u)$ and ${\cal R}$ are in general multi-sets.
\end{itemize}
\end{definition}

We point out that the pair $(\sigma, {\cal R})$ provide all the information we need about the items, therefore an instance of the Deadlock Safety Problem is encoded by a pair $\{(G,b), (\sigma, {\cal R})\}$; note in particular that two items that are sitting at a same vertex $v$ and have the same destinations and routes cannot be distinguished. Before moving on, we detail our definition of {\em path}: a path $P$ from a vertex $u$ to a vertex $v$ is a sequence of vertices $v_0, v_1, \ldots, v_{k-1}, v_k$ such that $v_0=u$, $v_k=v$ and $\{{v_{i-1},v_i}\}\in E$, for $i=1..k$. A path is {\em simple} if the vertices in the sequence are all distinct. 

\begin{remark}
While for the sake of simplicity we assume that each path defining a route is {\em simple}, our results and proofs do indeed extend to the case where routes are non-simple.
\end{remark}

Given an instance $\{(G,b), (\sigma, {\cal R})\}$ the Deadlock Safety Problem asks whether it is possible to deliver each item at its destination through a sequence of {\em feasible moves}:

\begin{definition}[Feasible move $\vec{P}$]\label{move}
Let $(\sigma, {\cal R})$ be a state, $u\in V$ a vertex with $\sigma(u) >0$ and $P$ a path in ${\cal R}(u)$. Let $v$ be the follower of $u$ with respect to $P$ and assume that $\sigma(v) < b(v)$. We may then associate with $P$ the feasible move $\vec{P}$ that shifts one item from $u$ to $v$. Through $\vec{P}$ we evolve from state $(\sigma, {\cal R})$ into a new state $\vec{P}(\sigma, {\cal R}) := (\sigma', {\cal R}')$. Namely:
\begin{itemize}
\item $\sigma'(u) = \sigma(u)-1$; $\sigma'(z) = \sigma(z)$, for each $z\notin\{u, v\}$;
\item $\sigma'(v) = \sigma(v) + 1$ if $v$ is not the destination of $P$, else $\sigma'(v) = \sigma(v)$;
\item ${\cal R}'(u) = {\cal R}(u)\setminus P$; ${\cal R}'(z) = {\cal R}(z)$ for each $z\notin\{u, v\}$;
\item ${\cal R}'(v) = {\cal R}(v)\cup \{P|v\}$ if $v$ is not the destination of $P$, else ${\cal R}'(v) = {\cal R}(v)$.
\end{itemize}
where $P|v$ is the sub-path of $P$ starting at $v$, i.e., if $P=v_0, v_1, \ldots, v_{k-1}, v_k$, with $u=v_0$ and $v=v_1$, $P|v=v_1, v_2,\ldots, v_{k-1}, v_k$.
\end{definition}

According to the previous definition, our model assumes that each item ``disappears'' as soon as it reaches its destination $w$, i.e., it does not occupy any slot of the buffer at $w$. Also with any path $P\in {\cal R}$ we may associate a feasible move that is denoted as ${\vec P}$. In the following, we also ``reverse'' this notation: if ${\vec P}$ is a feasible move, we let $P$ be the path defining that move.

\begin{definition}[Feasible sequence of moves ${\vec{\cal P}}$]\label{sequencemove}
Let $(\sigma^0, {\cal R}^0)$ be a state. A sequence ${\vec{\cal P}}:= ({\vec P}^1, \ldots, {\vec P}^k)$ is a feasible sequence of moves for $(\sigma^0, {\cal R}^0)$ if the followings hold, for each $i\in [k]$:
\begin{itemize}
\item $P^{i}$ is a path in ${\cal R}^{i-1}$;
\item ${\vec P}^{i}$ is a feasible move for the state $(\sigma^{i-1}, {\cal R}^{i-1})$;
\item $(\sigma^i, {\cal R}^i): = {\vec P}^i(\sigma^{i-1}, {\cal R}^{i-1})$.
\end{itemize}
In this case, we also denote the final state defined by ${\vec{\cal P}}$ as ${\vec{\cal P}}(\sigma^0, {\cal R}^0)$, i.e., ${\vec{\cal P}}(\sigma^0, {\cal R}^0):= {\vec P}^k(\ldots({\vec P}^1(\sigma^0, {\cal R}^0))\ldots)$.
\end{definition}

We are now ready to provide the definition of {\em strong deadlock sets}, a tool that is introduced in~\cite{blazewicz_1994} (even though they do not use this name). This requires another preliminary definition:

\begin{definition}[Follower]
Given a state $(\sigma, {\cal R})$, if for $u\in V$ there exists a path $P\in {\cal R}(u)$ such that $v$ is the vertex following $u$ on $P$, then $v$ is the follower of $u$ with respect to $P$, or simply a follower of $u$.
\end{definition}

\begin{definition}[Strong Deadlock Set]
Given a state $(\sigma, {\cal R})$, a non-empty subset $Q\subseteq V$ defines a strong deadlock set for $(\sigma, {\cal R})$ if:
\begin{itemize}
\item $\sigma(u) = b(u)$, for each $u\in Q$;
\item for each $u\in Q$, each follower of $u$ belongs to $Q$.
\end{itemize}
\end{definition}

If for a state $(\sigma, {\cal R})$ there is a strong deadlock set $Q$, then the items sitting at the vertices in $Q$ prevent each other from moving and are therefore stalled: it is therefore not possible to deliver each item at its destination. However, it is also possible that no set $Q$ provides a strong deadlock set for the $(\sigma, {\cal R})$ but nevertheless each possible feasible sequence of {\em moves} from $(\sigma, {\cal R})$ ends up in a state with a set of vertices that is a strong deadlock set. This leads to the following definition of {\em bound to deadlock} state and {\em safe} state.

\begin{definition}[Bound to Deadlock and Safe State~\cite{arbib_1988}]
A state ${(\sigma, {\cal R}})$ is bound to deadlock if every feasible sequence of moves starting at $(\sigma, {\cal R})$ ends up in a state with a strong deadlock set. A state that is not bound to deadlock is safe.
\end{definition}

If a state is safe then there exists a feasible sequence of moves to free the network. This goes as follows. First, we we associate with each state ${(\sigma, {\cal R}})$ a potential function:
\begin{equation}
\label{eq:potential_function}
u(\sigma, {\cal R}) = \sum_{u\in V}\sum_{P\in {\cal R}(u)} |E(P)|.
\end{equation}

Note that each feasible move will decrease the potential by one. Now suppose that $(\sigma, {\cal R})$ is a safe state and assume that $u(\sigma) >0$ (else $\sigma(v) = 0$ for each $v\in V$). Each state that has no strong deadlock sets admits at least one feasible move. By definition, for a safe state there exists a feasible sequence of moves that induces a sequence of states without strong deadlock sets. Hence this sequence of states has to end up in the state with potential 0, i.e., the network is free.
We may therefore state: 

\medskip\noindent
\textbf{The Deadlock Safety Problem} (\textsc{dsp}) . {\bf Given} a graph $G(V, E)$, sizes $b:V\mapsto {\cal Z}_+$, and a state $(\sigma, {\cal R}): \sigma(v)\leq b(v)$ for each $v\in V$; {\bf find} whether $(\sigma, {\cal R})$ is bound to deadlock or safe.

\begin{remark}\label{basicbis}
It is possible that $(\sigma, {\cal R})$ is safe but a sequence of moves from $(\sigma, {\cal R})$ takes into a state with a strong deadlock set.
\end{remark}

\subsection{Examples}

\paragraph*{Example}
\label{firstex}
An instance of \textsc{dsp} is given in Figure~\ref{fig:weak_not_sufficient_double}. For now ignore the blue arcs. The set $V: = \{A, B, C, D, E, F\}$ with $b(v)=2$ for $v\in \{A, B, C\}$ and $b(v)=1$ otherwise. Two items are stored at $A, B, C$, i.e., $\sigma(v)=2$ for $v\in \{A, B, C\}$, while $\sigma(D) = \sigma(E) = \sigma(F) = 0$. The edges of $E$ correspond to the black segments. As for ${\cal R}$, items at a same vertex have the same route, and they are respectively: $(A, E, B), (B, F, C), (C, D, A)$. The state ${(\sigma, {\cal R}})$ is safe: this goes as follows. E.g., move first one item from $A$ to $E$. Now one slot at the buffer at $A$ is empty and so we may take both items sitting at $C$ at their destination (one at a time). Now both slots at $C$ are empty and we may take the items sitting at $B$ at their destination (again, one at a time). We finally take to their destinations first the item sitting at $E$ and then that at $A$.

\begin{figure}
    \centering\includegraphics[scale=0.4]{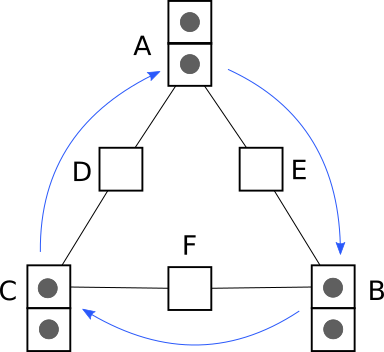}
    \caption{A \textsc{yes}-instance of \textsc{dsp}: see Example~\ref{firstex}. }
    \label{fig:weak_not_sufficient_double}
\end{figure}

\paragraph*{Example}
\label{secondex}
An instance of \textsc{dsp} is given in Figure~\ref{fig:weak_deadlock}. Ignore again the blue arcs. The set $V: = \{A, B, C, D, E\}$ with $b(v)=1$ for $v\in \{A, B, D, E\}$ and $b(C)=2$. The placement of items is such that $\sigma(A)=\sigma(E)=1$, $\sigma(C)=2$ and $\sigma(B)=\sigma(D)=0$. The edges of $E$ correspond to the black segments. As for ${\cal R}$, items must be routed according to the following paths: $(A, B, C, D, E), (C, B, A), (C, D, E), (E, D, C, B, A)$: note that the items at $C$ have different destinations and routes. It is easy to see that the state ${(\sigma, {\cal R}})$ is bound to deadlock.

\begin{figure}
    \centering\includegraphics[scale=0.4]{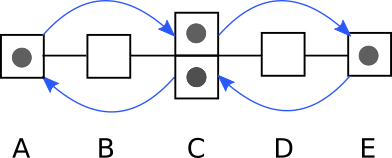}
    \caption{A \textsc{no}-instance of \textsc{dsp}: see Example~\ref{secondex}.}
    \label{fig:weak_deadlock}
\end{figure}

\section{Related literature}\label{literature}
The Deadlock Safety Problem falls in the domain of deadlock avoidance models (see the discussion at the beginning of Section\ref{intro}). In the following, we therefore focus on the previous literature for \textsc{dsp} as well as discuss some contributions that either deal with small variants of \textsc{dsp} or with other deadlock avoidance models mainly in the context of rail transportation.

\smallskip
{\bf Complexity results.} The first complexity result for \textsc{dsp} is given in~\cite{araki_1977}, in a setting where a system of multiple processes compete for the same computing resources. Besides other results, it is there shown that \textsc{dsp} is \textsc{np}-complete even when $b(v)=1$ for each $v\in V$.
That result is strengthened in~\cite{arbib_1988}, where it is shown that, when $b(v)=1$ for each $v\in V$, \textsc{dsp} is hard already on quite simple graphs, namely: grids, bipartite graphs and two-terminals series parallel graphs. Still in~\cite{arbib_1988} it is shown that things are a bit more involved when $G$ is a tree. Namely, if $G$ is a tree and $b(v)=1$ for each $v\in V$, then \textsc{dsp} is easy and a state is safe if and only if there are no {\em cycle of packets}~\cite{arbib_1988}, that are easy to recognize. However, they show that the problem is \textsc{np}-hard when $G$ is still a tree and $b\leq 3$.

\smallskip
{\bf Sufficient conditions.} Several papers deal with conditions that are sufficient for a state without strong deadlock sets to be safe. For our purposes, the most interesting contribution are given in~\cite{blazewicz_1994} where the authors show that, when $b(v)\geq 2$ for any $v\in V$, the \textsc{dsp} is easy and strong deadlock sets are the only obstructions (we will devote Section~\ref{bovetpaper} to this result). The same result is again shown in~\cite{reveliotis_1997}. 
Other papers, especially in the context of automated manufacturing systems, provide sufficient conditions that exploit suitable hypotheses on the routes of the items.
First~\cite{fanti_1997}, for the case  where $b(v) = 1$ for each $v\in V$, and~\cite{wu_2013} provide sufficient conditions that are based on properties of cycles in the Follower Network (see Definition~\ref{StrongNetwork}).  
Then in~\cite{kumar_1998} it is shown that a state without strong deadlock sets is safe as soon as each vertex $v$ with $b(v)=1$ is either preceded in all routes by a same vertex $v'$, or followed in all routes by a same vertex $v''$.
Finally Lawley and Reveliotis~\cite{lawley_2001} provide sufficient conditions when there is an ordering of the vertices in $V$ and the possible moves are consistent with this ordering according to a suitable definition.

\smallskip
{\bf Variants.}
Several papers in the literature deal with small variants of \textsc{dsp}, we name a few. First Toueg and Steiglitz~\cite{toueg_1981} investigate the complexity of the {\em Deadlock Exposure Problem}, i.e., the problem of deciding whether from a given starting state it is possible to reach a deadlock state or not. Then~\cite{bovet_1997} and~\cite{arbib_1988} deal with complexity of \textsc{dsp} when each item has a destination but not a fixed route.
Finally~\cite{diianni_1997} and~\cite{dalsasso_2022} deal with the complexity of \textsc{dsp} when some items (e.g. trains) are ``long'' and may therefore occupy more than one buffer slot at the same time, possibly on different buffers along the item route: this problem is shown to be \textsc{np}-complete in~\cite{diianni_1997} (even when items have destinations but not fixed routes); on the other hand, in~\cite{dalsasso_2022} it is proved to be easy when there are only two items such that each can occupy at most one slot in each buffer.

\smallskip
{\bf Deadlock avoidance in rail transportation.}
Deadlock avoidance problems often arise in the context of rail transportation. Most contributions in the literature approach these problems through heuristic algorithms and only a small number of exact algorithms have been proposed in literature. Heuristic solutions suffer either from the possibility of generating false positives~\cite{pachl_2007, li_2014}, i.e. deeming safe states as bound to deadlock and being therefore too restrictive, or false negatives~\cite{pachl_2011}, i.e. allowing wrong moves that may lead from a safe state to a bound to deadlock one.
Dal Sasso et al.~\cite{dalsasso_2021} provide \textsc{milp} formulations for \textsc{dsp} as well as for some recovery problem following a deadlock: these formulations have been suitably encoded and integrated inside an efficient Autonomous Traffic Management System that is able to deal with generic railway networks and long items, i.e., trains that are longer than individual railway sections (see above).
Another exact method for \textsc{dsp} with long items exploits a SAT-based reformulation of the problem and it is developed in Luteberget~\cite{luteberget_2021}.

\section{The case with buffers of size at least 2}\label{bovetpaper}

A characterization of strong deadlock sets as directed dominating sets in an auxiliary directed graph is given in~\cite{blazewicz_1994}. A {\em directed dominating set} in a directed graph $D$ is a set of nodes $S\subseteq V(D)$ such that for each vertex $u\in V(D)\setminus S$ there exists $v\in S$ such that $(u,v)\in E(D)$. In order to recall this characterization we need a couple of definitions.

\begin{definition}[Follower Network]\label{StrongNetwork}
Given a pair $(G, b)$ and a state $(\sigma, {\cal R})$, the (directed) Follower Network $D(\sigma, {\cal R})$ is defined as follows:
\begin{itemize}
\item $V(D(\sigma, {\cal R})) = V$;
\item $E(D(\sigma, {\cal R})) = \{(u,v): u,v\in V$ and $v$ is a follower of $u\}$.
\end{itemize}
\end{definition}

\begin{definition}[Free($\sigma$)]
Given a state $(\sigma, {\cal R})$, we let $Free(\sigma)$ be the set of vertices not saturated by $\sigma$, i.e., $Free(\sigma): = \{u\in V: \sigma(u) < b(u)\}$.
\end{definition}

\begin{lemma}\label{recognize}\cite{blazewicz_1994}
There is a strong deadlock set for a state $(\sigma, {\cal R})$ if and only if $Free(\sigma)$ is not a dominating set for the transitive closure of $D(\sigma, {\cal R})$. In particular, if $Free(\sigma)$ is not a dominating set, then the set of nodes that are not dominated by $Free(\sigma)$ defines a strong deadlock set for $(\sigma, {\cal R})$.
\end{lemma}

Note that we will later provide a proof for a lemma, Lemma~\ref{bisrecognize}, that generalizes Lemma~\ref{recognize}\footnote{Lemma~\ref{recognize} corresponds to Theorem 2 in~\cite{blazewicz_1994}. While most results in~\cite{blazewicz_1994} require that $b(u)\geq 2$ for each $u$, a careful check of the proof of Theorem 2 shows that it indeed holds for any $b$.}. Note also that following Lemma~\ref{recognize} strong deadlock sets can be found in polynomial time. We now deal with the case $b\geq 2$.

\begin{lemma}\label{carbovet}\cite{blazewicz_1994}
Let $\{(G, b), (\sigma, {\cal R})\}$ be an instance of \textsc{dsp} with $b\geq 2$. If $(\sigma, {\cal R})$ has no strong deadlock sets, then it is possible to find in polynomial time a feasible move $\vec{P}$, for some $u\in V$ and $P\in {\cal R}(u)$, such that ${\vec P}(\sigma, {\cal R})$ is again a state without strong deadlock sets.
\end{lemma}
Building upon Lemma~\ref{recognize} and Lemma~\ref{carbovet} one may therefore state the following: 

\begin{theorem}\label{fondcarbovet}\cite{blazewicz_1994}
Let $\{(G, b), (\sigma, {\cal R})\}$ be an instance of \textsc{dsp} with $b\geq 2$. The state $(\sigma, {\cal R})$ is bound to deadlock if and only if there is a strong deadlock set and this can be recognized in polynomial time.  
\end{theorem}

\section{Weak deadlock sets}\label{wds}

In order to go beyond the previous results we need to handle both vertices with $b(v) \geq 2$ and vertices with $b(v) = 1$.  The main intuition behind our results is that it is not ``wise'' to {\em leave} items on vertices with $b(v) = 1$. While in some cases this has to be done (we will provide an example) it is mostly wise that items only {\em transit} on such vertices and not {\em stay}. We start to grasp this idea with the following definition.

\begin{definition}[Wise Follower]
Let $(\sigma, {\cal R})$ be a state and $P$ a path in ${\cal R}(u)$ for some $u\in V:\sigma(u)>0$. A vertex $v$ is the {\em wise follower} of $u$ with respect to $P$, or simply a wise follower of $u$, if $v$ is the first vertex in $P$, $v \neq u$, that is either the destination of $P$ or such that $b(v)=1$ and $\sigma(v)=0$ do not both hold.
\end{definition}

When $b\geq 2$, the wise follower of $u$ with respect to a path $P\in {\cal R}(u)$ is the same as the follower of $u$ with respect to the same path. However that is not the case when $b(v)=1$, $v$ is free and $v$ is {\em not} the destination of $P$. In Figures~\ref{fig:weak_not_sufficient_double} and~\ref{secondex} the blue arcs point, for vertices $u$ with $\sigma(u)>0$, at their wise followers (note that in general a vertex may have more wise followers).

\smallskip
We are now ready to introduce our main new tool, the {\em weak deadlock set}. Weak deadlock sets can be thought as a relaxation of strong deadlock sets allowing to handle also some vertices $v$ with $b(v) = 1$. 

\begin{definition}[Weak Deadlock Set]
Given a pair $(G, b)$ and a state $(\sigma, {\cal R})$, a non-empty subset $Q\subseteq V$ defines a {\em weak deadlock set} for $(\sigma, {\cal R})$ if:
\begin{itemize}
\item $\sigma(u) = b(u)$, for each $u\in Q$;
\item for each $u\in Q$, each wise follower of $u$ belongs to $Q$.
\end{itemize}
\end{definition}

Weak deadlock sets are easy to recognize: this goes along the same lines as for strong deadlock sets so we postpone the details, see Section~\ref{recweak}. By now we observe that a strong deadlock set is always a weak deadlock set and that vice versa, when $b\geq 2$, a weak deadlock set is trivially strong. However, when $b$ is free a weak deadlock set is not necessarily strong: consider again the instance in Figure~\ref{secondex}: the set $\{A, C, E\}$ defines a weak deadlock set but not a strong one. 

The example in Figure~\ref{secondex} might suggest that a state with a weak deadlock set is bound to deadlock. However, that is not true, see Figure~\ref{fig:weak_not_sufficient_double}: the set $\{A, B, C\}$ defines a weak deadlock but the state is safe. Note that the graph in Figure~\ref{fig:weak_not_sufficient_double} is not a tree. As we show in the next theorem, that is not by chance.  

\begin{theorem}\label{cutlemma}
On trees, if for a state $(\sigma, {\cal R})$ there is a weak deadlock set, then the state is bound to deadlock.
\end{theorem}

\begin{proof}
Let $T(V,E)$ be a tree and $b:V(T)\mapsto {\cal Z}_+$. Let $Q$ be a weak deadlock set. We start with a claim about $Q$.
 
\smallskip
{\bf Claim:} {\em Either $Q$ is a strong deadlock set or each feasible move $\vec{P}$, for $P\in {\cal R}(u)$ and $u\in Q$, is such that in the state $\vec{P}(\sigma, {\cal R})$ there is again a weak deadlock.} {\bf Proof.} Let $u\in Q$ and $P\in {\cal R}(u)$ such that $\vec{P}$ is a feasible move ($Q$ is a strong deadlock if there are no such $u$ and $P$). Let $v$ be the follower of $u$ with respect to $P$. Since $Q$ is a weak deadlock set, it follows that $v$ is not the wise follower of $u$ with respect to $P$: so $\sigma(v) = 0$ and $b(v) = 1$.  For shortness, let $(\sigma', {\cal R}'):= \vec{P}(\sigma, {\cal R})$. Let $X$ and $\overline X$ the vertices of the connected components of the graph $T\setminus \{u, v\}$ (note that $T[X]$ and $T[\overline X]$ are both trees) with $u\in \overline X$ and $v\in X$. Note that $\sigma'(v) = 1$. Also note that $Q\cap X\neq\emptyset$, as the wise follower of $u$ with respect to $P$ is in $X$.

Now consider any vertex $z\in Q\cap X$ -- observe that ${\cal R}(z) = {\cal R}'(z)$ -- and consider a path $\Pi\in {\cal R}(z)$. Note that it is possible that the wise follower of $z$ with respect to $\Pi$ changes when moving from state $(\sigma, {\cal R})$ to state $(\sigma', {\cal R}')$. There are indeed two possibilities. In a former case, the wise follower with respect to $\Pi$ and  $(\sigma, {\cal R})$  was a vertex in $Q\cap X$: then either the wise follower stays the same also for $\Pi$ and state $(\sigma', {\cal R}')$ or the wise follower of $z$ with respect to $\Pi$ and $(\sigma', {\cal R}')$ is now $v$ (this happens if $v$ lies on $\Pi$ between $z$ and the previous wise follower). Otherwise, if for state $(\sigma, {\cal R})$ and $\Pi$ the wise follower of $z$ was a vertex in $Q\cap \overline X$, $v$ will be the wise follower of $z$ for $(\sigma', {\cal R}')$ and $\Pi$. 

In both cases, each wise follower of a vertex $z \in Q \cap X$ is in $\{v\} \cup (Q \cap X)$. Finally observe that, by construction, the (unique) wise follower of $v$ for $(\sigma', {\cal R}')$ is the vertex that was the wise follower of $v$ with respect to $P$ for $(\sigma, {\cal R})$ and such vertex belongs to $Q\cap X$. Therefore $\{v\}\cup (Q\cap X)$ is a weak deadlock set for  $(\sigma', {\cal R}')$. {\bf End of the proof of the claim.} 

Now suppose the statement of the lemma is not true, and assume in particular that $\{(T,b), (\sigma, {\cal R})\}$ defines a minimum counterexample with respect to the potential $u(\sigma, {\cal R})$ (see~\eqref{eq:potential_function}). Note that by hypothesis $(\sigma, {\cal R})$ is safe. Note also that we may assume that $\sigma(u) = 0$ for each $u\notin Q$ as otherwise we could remove just one item from $u$ and ${\cal R}(u)$ and define another instance $\{(T,b), (\sigma', {\cal R}')\}$, with $\sigma = \sigma - e_u$ and ${\cal R}'$ suitable, such that $(\sigma', {\cal R}')$ is still safe. We also claim that $Q$ still defines a weak deadlock set for the state $(\sigma', {\cal R}')$: note indeed that the wise follower $z$ of any $v\in Q$ with respect to some path $P\in {\cal R}(u)$ stays the same when moving from state $(\sigma, {\cal R})$ to state $(\sigma', {\cal R}')$. That is because $\sigma(u) > 0$ and $u\notin Q$ imply that the $u$ does not belong to the sub-path from $v$ to $z$ on $P$. But since $u(\sigma', {\cal R}') < u(\sigma, {\cal R})$, it follows that there is a contradiction. 

So we may assume that $\sigma(u) = 0$ for each $u\notin Q$. Therefore each feasible move $\vec{P}$ is such that $P\in {\cal R}(u)$ and $u\in Q$. In this case, following the claim, we first rule out that there exists a strong deadlock set since $(\sigma, {\cal R})$ is safe. But then any feasible move $\vec{P}$ is such that in the state $\vec{P}(\sigma, {\cal R})$ there is again a weak deadlock set. And this is a contradiction because at least one of these moves must lead to a safe state $((\sigma, {\cal R})$ is safe), and so there will be some state $\vec{P}(\sigma, {\cal R})$ that is again safe, has a weak deadlock set and has smaller potential than $(\sigma, {\cal R})$.
\end{proof}

\subsection{Recognizing weak deadlock}\label{recweak}
We now show how to recognize weak deadlock sets. We start with the following definition that generalizes Definition~\ref{StrongNetwork}.

\begin{definition}[Wise Follower Network]
Given a pair $(G, b)$ and a state $(\sigma, {\cal R})$, the Wise Follower Network $W(\sigma, {\cal R})$ is the directed network defined as follows:
\begin{itemize}
\item $V(\sigma, {\cal R}) = V$;
\item $E(W(\sigma, {\cal R})) = \{(u,v): u,v\in V$ and $v$ is a wise follower of $u\}$.
\end{itemize}
\end{definition}

Note that, if $b(v) \geq 2$ for each $v\in V$, then the Wise Network $W(\sigma, {\cal R})$ and the Wise Follower Network $D(\sigma, {\cal R})$ coincide. The next lemma provides a characterization for weak deadlock sets: it generalizes Lemma~\ref{recognize} and we provide a proof for the sake of completeness. Note that, following this characterization, weak deadlock sets can be recognized in polynomial time.

\begin{lemma}\label{bisrecognize}
There is a weak deadlock set for a state $(\sigma, {\cal R})$ if and only if $Free(\sigma)$ is not a dominating set for the transitive closure of the Wise Follower Network $W(\sigma, {\cal R})$. In particular, if $Free(\sigma)$ is not a dominating set, then the set of nodes that are not dominated by $Free(\sigma)$ defines a weak deadlock set for $(\sigma, {\cal R})$.
\end{lemma}

\begin{proof}
Let $W^+(\sigma, {\cal R})$ denote the transitive closure of $W(\sigma, {\cal R})$. Suppose that $Q$ is a weak deadlock set for $(\sigma, {\cal R})$. Then, by definition, $Q\neq \emptyset$ and $Q\cap Free(\sigma) = \emptyset$. Moreover, there is no arc $(u, v)$ in $W^+(\sigma, {\cal R})$ such that $u\in Q$ and $v\in Free(\sigma)$. Hence $Free(\sigma)$ is not a dominating set for $W^+(\sigma, {\cal R})$.

Vice versa let $X$ be the set of nodes that either belong to $Free(\sigma)$ or are dominated in $W^+(\sigma, {\cal R})$ by some node in $Free(\sigma)$. By definition, the set $V\setminus X$, if non-empty, is a weak deadlock set. 
\end{proof}

\section{Wise States}\label{ws}
In the previous section we pointed out that often it is not ``wise'' to {\em leave} items on vertices with $b(v) = 1$. This motivates the following definition:

\begin{definition}[Wise State]
A state $(\sigma, {\cal R})$ is {\em wise} if there is no vertex $v\in V$ such that $\sigma(v)=b(v)=1$. 
\end{definition}

The state depicted in Figure~\ref{fig:weak_not_sufficient_double} is wise while the state depicted in Figure~\ref{secondex} is not wise. Note that, if $b(v) \geq 2$ for each $v\in V$, then any state is wise. Recall also that, still in the case that $b(v) \geq 2$ for each $v\in V$, a weak deadlock set is strong. The following lemma is then a generalization of Lemma~\ref{carbovet}.

\begin{lemma}\label{carbovet_riv}
Let $\{(G, b), (\sigma, {\cal R})\}$ be an instance of \textsc{dsp} with $(\sigma, {\cal R})$ wise. If $(\sigma, {\cal R})$ has no weak deadlock sets, then it is possible to find in polynomial time a feasible sequence of moves ${\vec {\cal P}}$,  such that ${\vec {\cal P}}(\sigma, {\cal R})$ is again a state wise and without weak deadlock sets.
\end{lemma}

The proof of Lemma~\ref{carbovet_riv} goes along the same lines of the proof of Theorem 1 in~\cite{blazewicz_1994}: however a quite careful tuning is needed to suitably address wise states and weak deadlock sets. In particular, we need an additional lemma that gives more detail about the feasible sequence of moves ${\vec {\cal P}}$ in the statement of Lemma~\ref{carbovet_riv}. We state and prove this lemma before moving to the proof of Lemma~\ref{carbovet_riv}.

\begin{lemma}\label{wisesubpath}
Let $(\sigma, {\cal R})$ be a state and $u$ and $v\in V$ such that $v$ is the wise follower of $u$ with respect to a path $P\in {\cal R}(u)$ and $v\in Free(\sigma)$. Then there exists a feasible sequence of moves ${\vec {\cal P}}$ such that the followings hold:
\begin{itemize}
\item $\sigma'(u) = \sigma(u)-1$; $\sigma'(z) = \sigma(z)$, for each $z\notin\{u, v\}$;
\item $\sigma'(v) = \sigma(v) + 1$ if $v$ is not the destination of $P$, else $\sigma'(v) = \sigma(v)$;
\item ${\cal R}'(u) = {\cal R}(u)\setminus P$; ${\cal R}'(z) = {\cal R}(z)$ for each $z\notin\{u, v\}$;
\item ${\cal R}'(v) = {\cal R}(v)\cup \{P|v\}$ if $v$ is not the destination of $P$, else ${\cal R}'(v) = {\cal R}(v)$.
\end{itemize}
where we let $(\sigma', {\cal R}'):={\vec {\cal P}}(\sigma, {\cal R})$ and denote by $P|v$ is the path obtained from $P$ by removing the sub-path from $u$ to $v$.
\end{lemma}

\begin{proof}
Suppose that $v$ is the wise follower of $u$ with respect to some path $P\in {\cal R}(u)$, namely $P  = \{u\equiv u_0, u_1, \ldots, u_k \equiv v, u_{k+1}, \ldots, u_q\}$ for some $k\geq 1$ and $q\geq k$. Note that, if $k>1$, then each vertex $z\in \{u_1, \ldots, u_{k-1}\}$ is such that $b(z)=1$ and $\sigma(z)=0$. We may inductively associate with each edge $\{u_{i-1}, u_i\}$, $i=1..k$, a feasible move so as to define a feasible sequence of $k$ moves such that the above statements hold.
\end{proof}

We are now ready for the proof of Lemma~\ref{carbovet_riv} that for the sake of clarity we report from scratch.

\begin{proof}
Let $W^+(\sigma, {\cal R})$ denote the transitive closure of $W(\sigma, {\cal R})$. Since $(\sigma, {\cal R})$ has no weak deadlock set, it follows from Lemma~\ref{bisrecognize} that $Free(\sigma)$ is a dominating set for $W^+(\sigma, {\cal R})$. Since $(\sigma, {\cal R})$ is different from the empty state, it follows that there exist $u\in V$ and $v\in Free(\sigma)$ such that $v$ is the wise follower of $u$ with respect to some path $P\in {\cal R}(u)$. Note also that, since $(\sigma, {\cal R})$ is wise, it follows that $b(u)\geq 2$. Following Lemma~\ref{wisesubpath}, there exists a feasible sequence of moves ${\vec {\cal P}}$ such that the followings hold for the state $(\sigma', {\cal R}'):={\vec {\cal P}}(\sigma, {\cal R})$:
\begin{itemize}
\item $\sigma'(u) = \sigma(u)-1$; $\sigma'(z) = \sigma(z)$, for each $z\notin\{u, v\}$;
\item $\sigma'(v) = \sigma(v) + 1$ if $v$ is not the destination of $P$, else $\sigma'(v) = \sigma(v)$.
\end{itemize}

In the following, we let $W^+(\sigma', {\cal R}')$ denote the transitive closure of $W(\sigma', {\cal R}')$. By the definition of wise follower, $\sigma'$ is wise. Also, depending whether $v$ is the destination of $P$ or not, either $E(W(\sigma, {\cal R})) \oplus E(W(\sigma', {\cal R}')) \subseteq \textcolor{blue}{\{}\{u, v\}\textcolor{blue}{\{}$ or $E(W(\sigma, {\cal R}))\oplus E(W(\sigma', {\cal R}'))\subseteq \{\{u, v\}, \{v, w\}\}$ for some suitable $w\in V$. We now delve into two cases: $(i)$ $\sigma'(v) < b(v)$; $(ii)$ $\sigma'(v) = b(v)$.

$(i)$ Note that this happens if $\sigma(v) \leq b(v)-2$ or $v$ is the destination of $P$. In both cases, $Free(\sigma') = Free(\sigma) \cup \{u\}$. Note also that, since $b(u)\geq 2$, it follows that for each path $\Pi$ in $W(\sigma, {\cal R})$ that is using the arc $\{u, v\}$, the sub-path of $\Pi$ ``terminated'' at $u$ is a path in $W(\sigma', {\cal R}')$ (while this would not be the case if $b(u)= 1$, because of the definition of wise follower). Hence $Free(\sigma')$ is a dominating set for $W^+(\sigma', {\cal R}')$, and therefore $(\sigma', {\cal R}')$ has no weak deadlock set.

$(ii)$ In this case, $v\notin Free(\sigma')$. Following the discussion of the previous case, one easily realizes that we must deal with those vertices that in $W^+(\sigma, {\cal R})$ are dominated by $v$, possibly including $v$ itself. However, since $v$ is a wise follower of $u$ and $\sigma(v) = b(v)-1$, it follows that $b(v)\geq 2$ and $\sigma(v) > 0$. Therefore, there exists some vertex $z\in V$ such that $z$ is a wise follower of $v$ with respect to $(\sigma, {\cal R})$: note that $z$ is if a wise follower of $v$ also with respect to $(\sigma', {\cal R}')$. Since $Free(\sigma)$ was a dominating set for $W^+(\sigma, {\cal R})$ it follows that either $z\in Free(\sigma)$ or $(z,w)\in E(W^+(\sigma, {\cal R}))$ for some vertex $w\in Free(\sigma)$. We first suppose that $z\in Free(\sigma)$: in this case $z\in Free(\sigma')$ too, $Free(\sigma')$ is a dominating set for $W^+(\sigma', {\cal R}')$ and so $(\sigma', {\cal R}')$ has no weak deadlock set. The same conclusions hold if $(z,w)\in E(W^+(\sigma, {\cal R}))$ for some vertex $w\in Free(\sigma')$: note that in this case $w\neq v$, while it is possible that $w = u$.

Therefore, from now on we assume that $z\notin Free(\sigma)$ and that the only node $w\in Free(\sigma)$ such that $(z,w)\in E(W^+(\sigma, {\cal R}))$ is $v$ itself, as $v$ is the only vertex that belongs to $Free(\sigma)$ but not to $Free(\sigma')$. It follows that there exists $v'$ such that $\{v, v'\}\in E(W^+(\sigma, {\cal R}))$ and $\{v', v\}\in E(W(\sigma, {\cal R}))$, i.e., $v$ is a wise follower of $v'$ with respect to $(\sigma, {\cal R})$ (it is possibly the case that $v' = z$). Note that by definition $b(v')\geq 2$ and $\sigma(v') > 0$. But then we could have chosen $v'$ instead of $u$ (cf. the beginning of the proof) and defined $(\sigma', {\cal R}')$ consequently. Following again the proof above, it is easy to see that, with this choice, the corresponding state $(\sigma', {\cal R}')$ would have been still wise but also such that $Free(\sigma')$ is a dominating set for $W^+(\sigma', {\cal R}')$; and so $(\sigma', {\cal R}')$ would have had no weak deadlock set.

In order to complete the proof it remains to show that we always can find in polynomial time the sequence of moves ${\vec {\cal P}}$ we refer to in the statement. That is indeed the case as it is indeed enough to first choose a vertex $v$ such that $v$ a wise follower and $v\in Free(\sigma)$. Then it is enough to enumerate on the vertices $u$ such that $v$ is a wise follower of $u$ and check, for each such pair $(u, v)$, whether the sequence of moves ${\vec {\cal P}}$ ``induced'' by $(u,v)$ is such that $Free({\vec {\cal P}}(\sigma^0, {\cal R}))$ is a dominating set for the transitive closure of $W({\vec {\cal P}}(\sigma, {\cal R}))$.
\end{proof}

\begin{remark}\label{alsononwise}
A careful analysis of the above proof shows that Lemma~\ref{carbovet_riv} extends to the case where we do not ask $(\sigma, {\cal R})$ to be wise but simply such that no vertex $u$ with $\sigma(u)=b(u)=1$ is a wise follower of another vertex. When $u$ is the wise follower of a vertex $w$ and $\sigma(u)=b(u)=1$, then as soon as we move the item from $u$, then $u$ will be free but not anymore the wise follower of $w$: then we do not have any guarantee that the set of free vertices will be a dominating set for the transitive closure of $W(\sigma', {\cal R}')$.
\end{remark}

\begin{theorem}\label{fondcarbovet_riv}
Let $\{(G, b), (\sigma, {\cal R})\}$ be an instance of \textsc{dsp} such that $(\sigma, {\cal R})$ is a wise state. We can recognize in polynomial time whether $(\sigma, {\cal R})$ has weak deadlock sets and, if not, $(\sigma, {\cal R})$ is safe.
\end{theorem}

\begin{proof}
We can check whether $(\sigma, {\cal R})$ has weak deadlock sets in polynomial time according to Lemma~\ref{bisrecognize}. Suppose now that $\{(G, b), (\sigma, {\cal R})\}$ has no weak deadlock sets. According to Lemma~\ref{carbovet_riv} there exist a feasible sequence of moves ${\vec {\cal P}}$ such that the state ${\vec {\cal P}}(\sigma, {\cal R})$ is wise and without weak deadlock sets. Note also that ${\vec {\cal P}}(\sigma, {\cal R})$ has smaller potential than $(\sigma, {\cal R})$; it follows that if we iterate this argument and define a sequence of states that has to end up in the state with potential 0. Therefore $(\sigma, {\cal R})$ is safe.
\end{proof}

\smallskip
Differently from Theorem~\ref{fondcarbovet}, Theorem~\ref{fondcarbovet_riv} does not require that $b\geq 2$ but it requires that the starting state is wise. Moreover, while Theorem~\ref{fondcarbovet} provides conditions that are necessary and sufficient, Theorem~\ref{fondcarbovet_riv} provides only sufficient conditions for a wise state to be safe. There are indeed wise states $(\sigma, {\cal R})$ that have weak deadlock and nevertheless are safe, see again Figure~\ref{fig:weak_not_sufficient_double}. However, as we already discussed -- see Theorem~\ref{cutlemma} -- that does not happen on trees. We may therefore state the following:

\begin{theorem}\label{fondtrees}
Let $\{(T, b), (\sigma, {\cal R})\}$ be an instance of \textsc{dsp} such that $T$ is a tree and $(\sigma, {\cal R})$ is wise. The state $(\sigma, {\cal R})$ is bound to deadlock if and only if there is a weak deadlock set and this can be recognized in polynomial time.  
\end{theorem}

\begin{proof}
The proof goes along the same lines of the proof for Theorem~\ref{fondcarbovet_riv} but it also exploits Theorem~\ref{cutlemma}.
\end{proof}

Building upon Remark~\ref{alsononwise}, we may indeed extend the previous theorem as to deal with non-wise states with a particular structure, namely such that the condition $\sigma(v)=b(v)=1$ is allowed only at the leaves of the tree.

\begin{theorem}\label{fondtreesbis}
Let $\{(T, b), (\sigma, {\cal R})\}$ be an instance of \textsc{dsp} such that $T$ is a tree and $\sigma(v)=b(v)=1$ does not hold at any interval vertex. The state $(\sigma, {\cal R})$ is bound to deadlock if and only if there is a weak deadlock set and this can be recognized in polynomial time.  
\end{theorem}

\section{Concluding Remarks}

It is somehow surprising that the Deadlock Safety Problem is easy on general networks for $b\geq 2$ but it is \textsc{np}-hard already on simple graphs when $b=1$ and on trees when $b\leq 3$. In this paper, we shed some lights on this behaviour by means of two new tools, weak deadlock sets and wise states, that allow us to improve upon previous results from the literature. In particular, we show that any state that is wise and without weak deadlock sets is safe and, in this case, there exists a simple algorithm to free the network. We also characterize when in tree networks wise states are safe: namely, a wise state is safe if and only if it has no weak deadlock sets, and this can be recognized in polynomial time.

While these results contribute to give a better picture of the complexity of the Deadlock Safety Problem, we believe that we are still far from a complete picture. Consider the instance of \textsc{dsp} depicted in Figure~\ref{fig:obst}: the state is bound to deadlock but there are no weak deadlock sets. Therefore, even on a line network, there are non-wise states that are bound to deadlock but we do not know ``why''! Note that the instance is such that Theorem~\ref{fondtrees}, Theorem~\ref{fondtreesbis} and Remark~\ref{alsononwise} do not apply.

\begin{figure}
    \centering\includegraphics[scale=0.4]{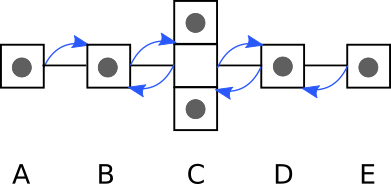}
    \caption{Vertices $A, B, C, D, E$ have respectively buffers of size  1, 1, 3, 1, 1 each holding one item but for $C$ that holds two item. The destination of the items in $A$ and $B$ and of one item in $C$ is $E$; the destination of the items in $D$ and $E$ and of the other item in $C$ is $A$. Blue arc provides the wise followers. There are no weak deadlock sets, but the instance is bound to deadlock.}
    \label{fig:obst}
\end{figure}

We therefore raise a few research questions:
\begin{enumerate}
    \item Which is the complexity of the Deadlock Safety Problem on line networks? The problem seems far from being trivial, even if we assume that there are only {\em left} items whose destination is the leftmost vertex on the line and {\em right} items whose destination is the rightmost vertex on the line, i.e., the setting of Figure~\ref{fig:obst}. 
    \item Which is the complexity of the Deadlock Safety Problem on trees when $b\leq 2$? Once again this is quite a relevant scenario in the context of rail transportation where networks are likely tree-like and either single or double-tracked~\cite{Hicks_2015}. 
    \item The Deadlock Safety Problem asks whether a state is safe or bound to deadlock. However, even knowing that a state is safe -- e.g. because $b\geq 2$ and there are no strong deadlock sets -- we might be interested in a {\em compact} representation of the feasible sequence of moves that frees the network, i.e., with a length that is polynomially bounded in the size of the input. Note that simply listing the moves does not do the job because their number (which is is equal to the potential of the starting state, see~\eqref{eq:potential_function}) needs not to be polynomially bounded in the size of the input e.g. because there is a large number of identical items. Is there any such representation?
\end{enumerate}

\smallskip\noindent
{\bf Acknowledgments} We thank Yuri Faenza for some useful comments on a previous draft of this manuscript.

\printbibliography
\end{document}